\documentclass{article}
\usepackage{amsmath,amsfonts,amssymb,amsthm}
\usepackage[british]{babel}
\usepackage{hyperref}
\usepackage{dsfont}

\usepackage{xcolor}

\usepackage{xfrac}

\newcommand{\bbR}{\mathbb{R}}

\newcommand{\eps}{\varepsilon}

\makeatletter
\newcommand{\vast}{\bBigg@{3}}
\newcommand{\Vast}{\bBigg@{5}}
\makeatother

\def\Tau{\mathcal{T}}

\allowdisplaybreaks

\newtheorem{Theorem}{Theorem}

\newtheorem{Proposition}{Proposition}
\newtheorem{Lemma}{Lemma}

\usepackage{scrextend}

\usepackage{lipsum}

\begin{document}

\date{\today}
\title{On the one dimensional cubic NLS in a critical space}
\title{On the one dimensional cubic NLS in a critical space}
\author{Marco Bravin \thanks{ BCAM - Basque Center for Applied Mathematics, Mazarredo 14, E48009 Bilbao, Basque Country - Spain } \\Luis Vega \footnotemark[1]}

\maketitle

\begin{abstract}
In this note we study the initial value problem in a critical space for the one dimensional Schr\"odinger equation with a cubic non-linearity  and under some smallness conditions. In particular the initial data is given by a sequence of Dirac deltas with different amplitudes but equispaced. This choice is motivated by a related geometrical problem; the one describing the flow of curves in three dimensions moving in the direction of the binormal with a velocity that is given by the curvature.  
\end{abstract}

\section{Introduction}

Since the work of Da Rios \cite{DaR}, it is known the connection between the evolution of a vortex filament in an incompressible inviscid fluid in three dimensions and the so called Binormal Curvature Flow (BF). The equation for this flow reads 
%
%
\begin{equation}
\label{VFE}
\partial_t \chi(t,s) = \partial_{s}\chi(t,s) \wedge \partial_{ss}^2 \chi(t,s)
\end{equation}  
where $ \chi: \mathbb{R} \times \mathbb{R} \longrightarrow \mathbb{R}^3 $ is a time dependent curve parametrized by the arclength $ s $. Note that the tangencial vector to the curve $ T(t,s) = \partial_{s} \chi (t,s) $ satisfies the so called Schr\"odinger map equation onto $ \mathbb{S}^2 $:
\begin{equation*}
\partial_t T = T \wedge \partial_{ss}^2 T.
\end{equation*} 
 Thanks to the Hasimoto transformation, see \cite{Has}, the above equation is associated with a focusing 1D cubic non-linear Schr\"odinger equation
\begin{equation*}
i\partial_t u + \Delta u + \frac{1}{2}(|u|^2-m(t)) u = 0,
\end{equation*}
for some $m(t)\in\mathbb{R}$. The freedom of choice of $m(t)$ comes form the gauge invariance of the underlying geometric problem, and plays a crucial role when the initial datum of \eqref{VFE} is singular. A specific kind of singularity that has recently received attention is the one given by the presence of corners. In particular, in \cite{scat:BV} and \cite{sta:BV}  the case of one corner and the stability of the solution is considered, and in \cite{Ban:Veg}  the more general situation of polygonal lines is studied. More concretely, in the latter work filaments made of, possibly infinity, segments of the same length $ 2\pi $ are analyzed\footnote{ The works \cite{J:S} and \cite{DH:V}, give numerical evidence of the interest of considering polygonal lines as initial datum for the vortex filament equation (VFE).}.
The corresponding initial datum for \eqref{VFE} translates to a, possibly infinite, sum of deltas of Dirac on $ 2\pi \mathbb{Z}$ as initial datum for the non-linear Schr\"odinger equation. The good choice of $m(t)$ in this situation turns out to be
\begin{equation}
\label{m(t)}
m(t)=\frac{M}{2\pi t},\qquad M = \sum_k |\alpha_{k}|^2.
\end{equation}
As a consequence one is led to study the Initial Value Problem (IVP)
\begin{align}
\label{CNLS:sys}
\partial_t u = \, & i\left(\Delta u + \frac{1}{2}\left(|u|^2-\frac{M}{2\pi t}\right) u \right)
\\
u(0,.) = \, & \sum_{k} \alpha_{k}\delta_{k}(.),\qquad  \delta_{k}(x)=\delta(x-k).
\end{align}
%

%

In \cite{Kita}, the IVP analogous to \eqref{CNLS:sys} but for subcritical nonlinearities, $ u|u|^{p-1} $, $ p < 3 $, is studied, and well-posedness is proved among all the solutions that are written as the sum
\begin{equation}
\label{ansatz}
u(t,x) =  \sum_{k} A_k(t) e^{it\Delta} \delta_{k}(x).
\end{equation}
Above $e^{it\Delta} $ stands for the usual free propagator of the linear Schr\"odinger equation. As done in  \cite{Ban:Veg}, one can use the identity $e^{it\Delta} \delta_{k}= \frac{e^{i\frac{(x-k)^2}{4t}}}{\sqrt {t}}$  to obtain that $u$ can be written  for $t>0$  in terms of a new function $v$ where  $u=\mathbb{T}(v)$. Here $\mathbb{T} $ is the so-called pseudo-conformal transformation that is given by
\begin{equation*}
\mathbb{T} (v)(t,x)=\frac{e^{i\frac{x^2}{4t}}}{\sqrt {t}}\,\overline v\left(\frac 1t,\frac x{t}\right).
\end{equation*}
As a consequence the corresponding non-linear potential becomes
$$|u|^{p-1} =\frac{1}{t^{\frac{p-1}{2}}} |v|^{p-1},$$
so that the factor $\frac{1}{t^{p-1}}$ is locally integrable around $t=0$ if and only if $p<3$. In this case $m(t)$ can be chosen identically zero while for $p=3$ a modification is needed, (cf. \cite{Ban:Veg}). In this particular situation the equation for $v$ turns out to be
\begin{equation}\label{vNLS}
iv_t+v_{xx}+\frac 1{2t}(|v|^2- M)v=0,
\end{equation} 
with $M$ as in \eqref{m(t)}.

Notice that in the definition of $\mathbb{T} $ there is an inversion of the time variable and therefore the IVP for $u$ becomes a scattering problem for $v$. Also observe that the solutions of  \eqref{vNLS} formally preserve the $L^2$ norm
\begin{equation}\label{L2}\int |v|^2\, dx.
\end{equation} 
Due to the fact that from \eqref{ansatz}
$$v(t,x)=\sum_k A_k(t)e^{ikx},$$
we immediately obtain that
\begin{equation}\label{l2}
\sum_k|A_k(t)|^2=\sum_k|\alpha_k|^2=M,
\end{equation} 
which justifies the choice done in \eqref{m(t)}.
Also
\begin{equation*}
E(v)(t):=\frac12\int|v_x(t)|^2\,dx-\frac{1}{4t}\int(|v|^2-|\alpha|^2)^2\,dx
\end{equation*}
satisfies 
$$\partial_tE(v)(t)=\frac{1}{4t^2}\int(|v|^2-|\alpha|^2)^2\,dx.$$

Once the ansatz \eqref{ansatz} is fixed, the IVP  \eqref{CNLS:sys} is reduced to solve the following  infinite dynamical system in  the variables $ A_k(t) $
 \begin{equation}
 \label{Ak}
i\partial_{t} A_{k} = \frac{1}{8\pi t} \sum_{j_1,j_2,j_3 \in NR_k} e^{-i\frac{|k|^2-|j_1|^2+|j_2|^2-|j_3|^2}{4t}}A_{j_1}\bar{A}_{j_2}A_{j_3} - \frac{1}{8 \pi t}|A_{k}|^2A_{k},
\end{equation}
where we have used the notation
$$ NR_k = \{ (j_1,j_2,j_3) \in \mathbb{Z}^3\,  \text{such that}\,  k = j_1-j_2 -j_3\,\, \text{and}\,\, k^2-j_1^2 +j_2^2-j_3^2 \neq 0 \}.$$
The lack of integrability at the origin of  the $1/t$ factor in the above expression requires the use of the phase renormalization 
\begin{equation}
\label{tilde}
A
_{k}(t) = e^{i\frac{|\alpha_{k}|^2}{8 \pi}\log t} \tilde{A}_{k}(t).
\end{equation} 
Then for $ t> 0$, the functions $ \tilde{A}_{k} $ satisfy the system
\begin{equation}
\label{last:ban:veg}
i\partial_{t} \tilde{A_{k}} = \frac{1}{8\pi t} \sum_{j_1,j_2,j_3 \in NR_k} e^{-i\frac{|k|^2-|j_1|^2+|j_2|^2-|j_3|^2}{4t}}e^{-i\frac{|\alpha_k|^2-|\alpha_{j_1}|^2+|\alpha_{j_2}|^2-|\alpha_{j_3}|^2}{8\pi}\log t}\tilde{A}_{j_1}\bar{\tilde{A}}_{j_2}\tilde{A}_{j_3} - \frac{1}{8 \pi t}\left(|\tilde{A}_{k}|^2 - |\alpha_k|^2\right)\tilde{A}_{k},
\end{equation} 
with the initial condition
\begin{equation}
\label{iv}
\tilde{A}_{k}(0) = \alpha_k.
\end{equation}
If we define the spaces $ l^{2,s}$ as those sequences $ \{\alpha_k \} $ such that
$$\|(\alpha_k)\|_{l^{p,s}} = \left( \sum_k |(1+k^2)^s \alpha_k|^p\right)^{1/p}<+\infty$$
with $l^{p,0}=l^{p}$, then it is proved in  \cite{Ban:Veg} that the system \eqref{last:ban:veg}-\eqref{iv} is local-in-time well-posed for $ \{\alpha_k \} \in l^{2,s} $ with $ s > 1/2 $ and for $ \{\alpha_k \} \in l^{1} $.  Observe that due to \eqref{tilde} if the IVP \eqref{last:ban:veg}-\eqref{iv} is well posed then the ones associated to \eqref{Ak} and to \eqref{CNLS:sys} is ill posed, cf. Theorem 1.4 in  \cite{Ban:Veg}.

In this work we would like to consider initial data  in $ l^p $ for $p \in (1,\infty) $. We will perform a fixed point argument based on considering the regularity of the coefficients $A_k(t)$. Due to the fact that we are working directly with solutions to the system \eqref{last:ban:veg} it will be enough for our purposes to use the classical Sobolev spaces instead of the usual ones in this setting introduced by J. Bourgain in \cite{Bou:1}. 

The results of this paper regarding the IVP \eqref{last:ban:veg}-\eqref{iv} can be resumed as follows. \footnote{In Theorem  \ref{Theo:p=2} the result is stated for $ B_k(t/4) = \tilde{A}(1/t) $.}

\begin{itemize}

\item For $  p \in (1,+\infty) $, 

\begin{itemize}
	
	\item (i) Large time well-posedness for small initial datum. More precisely for any $ T > 0 $, there exists $ \eps(T) > 0 $ such that if the $ l^p $ norm of the initial datum $ \{\alpha_{k}\} $ is smaller then $ \eps(T) $, then there exists a unique solution of \eqref{last:ban:veg} in $ [0, T] $ in an appropriate sense. 
	
\item (ii) Short-in-time well-posedness for small enough initial datum in $ l^{\infty}$. More precisely if the $ l^{\infty} $ norm of $\{ \alpha_k \}$ is small enough then there exits a small time $ T(\|\alpha\|_{\l^{\infty}},\|\alpha\|_{\l^p}) $ such that a unique solution of \eqref{last:ban:veg} exists in $ [0,T] $ in an appropriate sense.  
	
\end{itemize} 	

\item    For $ p = 2 $, global in time well-posedness with small assumption in $ l^{\infty} $ for the initial datum. As it can be expected this result follows from (ii) and the $ l^2 $ conservation law.

\end{itemize}

The first smallness condition is rather natural due to the criticality of the problem and that the method of proof is perturbative. Also the fact that the potential $1/t$ is not bounded for $t>0$ makes necessary to consider $T < + \infty $. The second smallness condition is less natural. It comes from the linear terms in \eqref{equ:til:R} below. This difficulty already appears in the case of the perturbation of one delta function and was addressed in \cite{scat:BV} and \cite{sta:BV}. In these papers the linear terms are not considered as perturbative. Nevertheless, there is a price to be paid for it, namely the possible growth of the zero Fourier mode of the perturbation (cf. section 6 in \cite{sta:BV} and Appendix B in \cite{scat:BV}). In this paper we consider another type of perturbation,  a periodic one, and this growth can not be possible due to the conservation law \eqref{L2}-\eqref{l2}. The use of this conservation law in the linearized system, cf. with (55) and (56) in  \cite{sta:BV} in the case of one delta, will require a delicate analysis that we postpone for the future.

Finally, although it is not clear that the above results extend to initial data in $ l^{\infty} $ we are able to exhibit some particular solutions that include the case of regular polygons for \eqref{VFE}, see \cite{ J:S} and \cite{DH:V}. More precisely, if the initial condition is a constant sequence, in other words $ \alpha_k = \alpha $ for any $ k \in \mathbb{Z} $, then there exists an explicit solution to the problem. The result is given in Proposition \ref{prop:exa} that is proved in collaboration with V. Banica. The uniqueness of this solution is quite likely a very challenging problem.
 
\subsection{Connection with previous results} 

The well-posedness of 1D cubic NLS on the line and on the circle was firstly tackled respectively in \cite{T} and \cite{Bou:1} for $ L^2 $ initial data. A natural scaling for the equation is $ u_{\lambda}(t,x) = \lambda u(\lambda^2 t, \lambda x) $. Therefore the $L^2(\Bbb R)$ setting is far from being scaling invariant, and something similar happens in the periodic case if the independence of the time of existence  on the size of the period is included in the analysis. The first steps to go beyond the $L^2(\Bbb R)$ theory were given in \cite{VV} in a functional setting constructed adhoc using the well known Strichartz estimates. This approach  was extended later on in \cite{G} where well-posedness is studied in the Fourier-Lebesque spaces that we denote by $ \mathcal{F}L^p $. These are spaces of distributions whose Fourier transforms are bounded in $ L^p(\Bbb R) $. Therefore, they are invariant respect to translations  in the phase space. Moreover $ \mathcal{F}L^{\infty} $ is also scaling invariant and therefore critical.   Local in time well-posedness, also including periodic boundary conditions using the Fourier coefficients instead of the Fourier transform,  was shown in \cite{GH} in the $ \mathcal{F}L^p $ space for $ 2 < p < + \infty $.

At the same time plenty of effort has been done to understand well-posedness in the setting of Sobolev spaces. In this case the homogeneous space $ \dot{H}^{-1/2} $ is critical. Many papers have been devoted to clarify what happens between $ L^2 $ and $ \dot{H}^{-1/2} $. In particular, it has been shown ill-posedness, in the sense that a data to solution map which is uniformly continuous does not exist in $ H^{s} $ with $ s < 0 $. Even more, an inflation phenomena concerning the growth of the Sobolev norms has been proved, see \cite{KPV01}-\cite{CCT}-\cite{CK}-\cite{K}-\cite{Oh}. Finally in \cite{HKV} it has been shown well-posedness in $ H^{s} $ for $ s > -1/2 $.   This result  provides a weaker notion of continuity for the data to solution map. Moreover, the connection with this result and the Fourier-Lebesgue spaces has been recently done in \cite{Oh-Wa}.

In this paper, we are considering sums of Dirac's deltas as initial datum. Note that a Dirac delta is critical for the $ \dot{H}^{-1/2} $ and belongs to $ \mathcal{F}L^{\infty} $ space. This means that we are considering a critical regime. 

As said before, in \cite{Ban:Veg} the weighted spaces $l^{2,s} $ for $ s > 1/2 $ are used. The main reason is that they are very convenient to obtain solutions of \eqref{VFE} from those of \eqref{CNLS:sys}. Up to what extent solutions of \eqref{VFE} can be constructed using the ones obtained in this paper seems to us a very challenging question that we propose to address in the future.

\paragraph{Acknowledgements.} We want to thank V. Banica for very fruitful discussions and more concretely for her contribution in Proposition \ref{prop:exa}. 
Marco Bravin  is supported by ERC-2014-ADG project HADE Id. 669689 (European Research Council).
Luis Vega is supported by ERC-2014-ADG project HADE Id. 669689 (European Research Council), MINECO grant BERC 2018-2021, PGC2018-094522-B-I00 and  SEV-2017-0718 (Spain).

\section{Well-posedness results for initial datum in $ l^p $ with $ p < + \infty$}

To present the result we rewrite the equations in a fixed point framework. We start by introducing some notation. 

Let us recall that $ NR_k = \{ (j_1,j_2,j_3) \in \mathbb{Z}^3 $ such that $ k = j_1-j_2 -j_3 $ and $ k^2-j_1^2 +j_2^2-j_3^2 \neq 0 \}$. Moreover, for any triple in $ NR_{k} $, it holds $ m := k^2-j_1^2 +j_2^2-j_3^2 = 2(k-j_1)(j_1-j_2) $. In particular the map $ (j_1,j_2,j_3) \to (k^2-j_1^2 +j_2^2-j_3^2, k-j_1) := (m,z) $ is a bijection between $ NR_{k} $ and $ \{ (m,z) \in \mathbb{Z}^2 $ such that $ m \neq 0 $ and $ 2z $ divide $ m \}$. We finally denote by 
\begin{equation}
\label{def:r(m)}
r(m) = \{ z \in \mathbb{Z} \text{ such that } 2z \text{ divide } m  \}.
\end{equation}

The system \eqref{last:ban:veg} can be rewritten as 
\begin{equation*}
i\partial_{t} \tilde{A_{k}} = \frac{1}{8\pi t} \sum_{m \neq 0}\sum_{r(m)} e^{-i\frac{m}{4t}}e^{-i\frac{\Lambda_m}{8\pi}\log t}\tilde{A}_{j_1}\bar{\tilde{A}}_{j_2}\tilde{A}_{j_3} - \frac{1}{8 \pi t}\left(|\tilde{A}_{k}|^2 - |\alpha_k|^2\right)\tilde{A}_{k},
\end{equation*} 
where we denote $ \Lambda_m = |\alpha_k|^2-|\alpha_{j_1}|^2+|\alpha_{j_2}|^2-|\alpha_{j_3}|^2 $.
Let us introduce the new variable
$$ B_k(t/4) = \tilde{A}_k(1/t).$$
The equations for $ B_k $ are 
\begin{equation}
\label{B:equ}
i\partial_{t} B_{k} = -\frac{1}{8\pi t} \sum_{m \neq 0}\sum_{r(m)} e^{-imt}e^{i\frac{\Lambda_m}{8\pi}\log 4t}B_{j_1}\bar{B}_{j_2}B_{j_3} + \frac{1}{8 \pi t}\left(|B_{k}|^2 - |\alpha_k|^2\right)B_{k} \quad  \text{ and } \quad \lim_{t \to +\infty} B_k = \alpha_k.
\end{equation}  
Writing $ B_k = R_k + \alpha_k $ and integrating the above equations, we get 
\begin{equation}
\label{equ:til:R}
R_{k}(t) = \frac{i}{8 \pi}\int_{t}^{+\infty} \frac{1}{\tau} \sum_{m \neq 0}\sum_{r(m)} e^{-i m\tau}e^{i\frac{\Lambda_m}{8\pi}\log 4 \tau}(R_{j_1}+\alpha_{j_1})\overline{(R_{j_2}+\alpha_{j_2})}(R_{j_3}+\alpha_{j_3}) - \frac{1}{\tau}\left(|R_{k}+\alpha_k|^2 - |\alpha_k|^2\right)(R_{k} +\alpha_k) \, d\tau.
\end{equation}

We show well-posedness of solutions for the system \eqref{equ:til:R} via a fixed point argument. More precisely we introduce a map $ \Tau $ in  \eqref{equ:Tau:R} and we show that it is a contraction on some bounded subset of the Banach space $ X^{s}_p $. To introduce this space let us start by defining the Banach spaces $ \tilde{H}^{s,p}$ and $ \mathcal{H}^{s}_{p} $. 
%

At an informal level, for a bounded interval $ I \subset \bbR $, the space $ \mathcal{H}^{s}_{p}(I) $ corresponds to a Fourier-Lebesque space for non periodic functions. We say that $ f $ is an element of $ \mathcal{H}^{s}_{p}(I) $ if it admits an extension $ \tilde{f} $ defined on a bigger interval $ I^e \supset I $ such that $ \tilde{f} $ in a Fourier-Lebesque-type space $ \tilde{H}^{s,p}(I^e) $. Then, the $ \mathcal{H}^{s}_{p} $ norm is defined as the infimum of the $ \tilde{H}^{s,p} $ norm over all the possible extensions.

Let us recall that for $ p \in[1,+\infty) $, $ s \in \bbR $ and $ \Pi_4 = \sfrac{\bbR}{4\pi \mathbb{Z}}  $ the torus of size $ 4\pi $, the Fourier-Lebesque space $ \mathcal{FL}^{s,p}(\Pi_4)$ is defined by the norm 
\begin{equation*}
\| \mathfrak{f} \|_{\mathcal{LF}^{s,p}(\Pi_4)}^p = \sum_{2k \in \mathbb{Z}} \left| \langle k \rangle^{s} \hat{\mathfrak{f}}_{k}\right|^p,
\end{equation*}
where $ \hat{\mathfrak{f}}_k $ denotes the $ k $-th Fourier coefficient of $ \mathfrak{f} $.

For $ \nu \in \mathbb{N} $, let $ I_{\nu} = [\pi \nu,  \pi (\nu+2)] $ and $ I^e_{\nu} = [\pi(\nu-1), \pi (\nu+3)]$. For $ s \in \mathbb{R}^+ $ and $ p\in [1, +\infty)$, we denote by $ \tilde{H}^{s,p}(I^e_{\nu}) $ the Banach space of functions $ \tilde{f}:  I^e_{\nu}  \longrightarrow \mathbb{C} $ with norm given by
\begin{equation*}
\| \tilde{f} \|_{\tilde{H}^{s,p}(I_{\nu}^e)}^p = \sum_{2k \in \mathbb{Z}} \left| \langle k \rangle^{s} \hat{\tilde{f}}_{k}\right|^p,
\end{equation*}
where $ \hat{f}_k $ denotes the $ k $-th Fourier coefficient of $ \tilde{f}$ as a $ 4\pi $-periodic function. Notice that the $ \tilde{H}^{s,p} $ space is isomorphic to the Fourier-Lebesque space $ \mathcal{FL}^{s,p}(\Pi_4) $ through the map that sends $ \tilde{f} $ to its $ 4 \pi $ periodic extension. For $ s = 0 $, we use the notation $ \tilde{L}^p(I_{\nu}^e) $ instead of  $ \tilde{H}^{0,p}(I_{\nu}^e) $.

We denote by $ \mathcal{H}^{s}_p(I_{\nu}) $ the Banach space of functions $ f: I_{\nu} \to \mathbb{C} $ with norm given by
\begin{equation*}
\|f\|_{\mathcal{H}^{s}_{p}(I_{\nu})}^p = \inf_{\substack{\tilde{f}\in \tilde{H}^{s,p}(I^{e}_{\nu}),  \\ \tilde{f}|_{I_{\nu}}  = f}}  \| \tilde{f} \|_{\tilde{H}^{s,p}(I_{\nu}^e)}^p = \inf_{\substack{\tilde{f}\in \tilde{H}^{s,p}(I^{e}_{\nu}),  \\ \tilde{f}|_{I_{\nu}}  = f}} 
\sum_{2k \in \mathbb{Z}} \left| \langle k \rangle^{s} \hat{\tilde{f}}_{k}\right|^p.
\end{equation*}
Let us notice that for $ s > 1/p $ the space $ \tilde{H}^{s,p}(I_{\nu}^e) \subset C^0(\Pi_4) $. This implies that for elements $ \tilde{f} $ of $  \tilde{H}^{s,p}(I_{\nu}^e) $ with $ s > 1/p $, it holds $ \tilde{f}(\pi(\nu-1)) =\tilde{f}(\pi(\nu+3)) $. At the contrary a function in $ \mathcal{H}^s_p(I_{\nu}) $ for $ s > 1/p $ is continuous in $I_{\nu}$, it admits an extension in $ C^0(\Pi_4)$  but it does not need to be $ 2\pi$-periodic in the sense that in general $ f(\pi \nu ) \neq f(\pi(\nu +2)) $.

As before we denote by $ \hat{L}^p(I_{\nu}) $ the space $ \mathcal{H}^0_p $. Moreover   the homogeneous seminorm $ \dot{\mathcal{H}}^{s}_{p} $ is given by  
\begin{equation*}
\|f\|_{\dot{\mathcal{H}}^{s}_{p}(I_{\nu})}^p = \inf_{\substack{\tilde{f}\in \tilde{H}^{s,p}(I^{e}_{\nu}),  \\ \tilde{f}|_{I_{\nu}}  = f}} 
\sum_{2k \in \mathbb{Z}} \left|  |k|^{s} \hat{\tilde{f}}_{k}\right|^p.
\end{equation*}

For  sequences $ \{ R_{k} \}_{k \in \mathbb{Z}} $ of functions $ R_k : (0, +\infty) \to \mathbb{C} $, we introduce the norm 
\begin{equation*}
\left\| \{R_k\}_{k \in \mathbb{Z}} \right\|_{X^{s}_p}^p = \sup_{\nu \geq 0} \|(\nu+1) \{R_k\}_{k \in \mathbb{Z}} \|_{l^p(\mathbb{Z};\mathcal{H}^{s}_{p}(I_{\nu}))}^p = \sup_{\nu \geq 0} \sum_{k \in \mathbb{Z}} \|(\nu+1) R_k \|_{\mathcal{H}^{s}_{p}(I_{\nu})}^p 
\end{equation*}
and $ X^{s}_p $ the space of sequences $\{R_k\}_{k\in \mathbb{Z}} $ for which the norm $ \|\{R_k\}_{k\in \mathbb{Z}}\|_{X^{s}_p}$ is bounded. In the following we use $ R $ to denote a sequence $ \{R_k\}_{k \in \mathbb{Z}} $ to simplify the notation and we use subindexes to denote the elements of a sequence.
Finally for any $ T > 0 $, we define the space  $ X^{s}_p(T) $ the space of sequences $ R $ of functions $ R_k : (T; +\infty) \to \mathbb{C} $ for which the norm 
$$ \left\| R \right\|_{X^{s}_{p}(T)} = \inf_{\substack{\tilde{R}\in X^{s}_p \\
		\tilde{R}|_{(T,\infty)}= R}} \|\tilde{R}\|_{X^{s}_p}$$
is bounded.
Let us conclude this section with the main result.

\begin{Theorem}
\label{Theo:p=2}
	
Let $ \{ \alpha_k \}_{k\in \mathbb{Z}} \in l^p $ an initial data for $ p \in (1,\infty)$.

\begin{itemize}
	
\item (Large in time well-posedness) For any $ T > 0 $, there exists $ \delta $ such that if $ \| \alpha \|_{l^p} < \delta $, then there exists a unique solution  $ R $ to \eqref{equ:til:R} in $ X^{ s}_{p}(T) $ for some $ s $ with $ \max \{ 1/p, 1- 1/p \}  < s < 1 $.

\item (Short in time well-posedness) There exists $ \delta > 0 $ such that if $ \| \alpha \|_{l^{\infty}} < \delta $, then for a big enough time $  T  \gg 1 $ there exists a unique solution  $ R $ to \eqref{equ:til:R} in $ X^{ s}_{p}(T) $  for some $ s $ with $ \max \{ 1/p, 1- 1/p \} < s < 1 $.
	
\end{itemize}	
Moreover in the special case $ p = 2 $, if $ \| \alpha \|_{l^{\infty}} < \delta $ for $ \delta $ small enough, then there exists a unique solution to \eqref{equ:til:R} in $ (0, +\infty) $, which belongs to $ \bigcap_{T > 0 } X^{ s}_{p}(T) $ for some $  1/2 < s < 1 $.
	
\end{Theorem}

We postpone the proof to Section \ref{Sec:proof}.

\subsection{Example of explicit solutions for special $ l^{\infty} $ initial datum }

Although we are not able to extend the previous results to the case $ p= \infty $, there exist non-trivial solutions to the system \eqref{last:ban:veg} with initial datum in $ l^{\infty}\setminus \bigcup_{p\geq 1} l^p $. In particular we are able to exhibit an explicit solution in the case where the initial datum is a constant sequence, in other words the $ \alpha_k = \alpha \neq 0 $ for any $ k \in \mathbb{Z} $.

This observation has been made in collaboration with V. Banica who has kindly allowed us to include it in this article.

For consistency let us present the results in the $ B $ variables.

\begin{Proposition}
\label{prop:exa}
Given 	$ \alpha \in \mathbb{C} $, then there exists an explicit solution to \eqref{B:equ} with initial datum $ \alpha_k = \alpha $ for any $ k \in \mathbb{Z} $. Namely the solution is
\begin{equation*}
B_k(t) = B(t) =  \alpha e^{-i|\alpha|^2\int_t^{+\infty} \sum_{m \neq 0}\sum_{r(m)} \frac{e^{-im \tau}}{8\pi \tau} \, d\tau}, 
\end{equation*}
where $ r(m) $ is defined in \eqref{def:r(m)}.	
\end{Proposition}

We postpone the proof of this result in Section \ref{sec:ex}. 

\section{Proof of Theorem \ref{Theo:p=2}.}
\label{Sec:proof}

The proof is based on a fixed point argument in the space $ X^{ s}_p $. Let $ \eta \in C^{\infty}(\bbR) $ such that $ 0 \leq \eta \leq 1 $, $ \eta = 0 $ in an open neighborhood of $ (-\infty,0] $, $ \eta(t) = 1 $ for $ t \geq \pi $. Moreover for $ N \in \mathbb{N} $ we define $ \eta_N(t) = \eta(t-\pi N) $. We consider the map $ \Tau_{\eta_N}: X^s_{p} \longrightarrow X^{s}_{p} $ with $ k $-th component given by
\begin{align}
\left(\Tau_{\eta_N} (R)(t)\right)_k = &  \frac{i\eta_N(t)}{8 \pi}\int_{t}^{+\infty} \frac{\eta_N(\tau)}{\tau} \sum_{m \neq 0}\sum_{r(m)} e^{-i m\tau}e^{i\frac{\Lambda_m}{8\pi}\log 4 \tau}(R_{j_1}+\alpha_{j_1})\overline{(R_{j_2}+\alpha_{j_2})}(R_{j_3}+\alpha_{j_3}) \, d\tau  \nonumber \\ 
 & \, - \frac{i\eta_N(t)}{8 \pi}\int_{t}^{+\infty} \frac{\eta_N(\tau)}{\tau}\left(|R_{k}+\alpha_k|^2 - |\alpha_k|^2\right)(R_{k} +\alpha_k) \, d\tau. \label{equ:Tau:R}
\end{align} 
We start by presenting some bounds of  the map $ \Tau_{\eta_N} $. To do that we rewrite 
\begin{equation*}
(\Tau_{\eta_N}(R))_k = (\Tau_{\eta_N,0}(R))_k + (\Tau_{\eta_N,1}(R))_k + (\Tau_{\eta_N,2}(R))_k
\end{equation*}
where $ \Tau_{\eta_N,0} $ is independent on $ R $, $ \Tau_{\eta_N,1} $ is linear and $ \Tau_{\eta_N,2} $ is super-linear. For example
\begin{equation*}
(\Tau_{\eta_N,0}(R))_k =  \frac{i\eta_N(t)}{8 \pi} \int_{t}^{\infty} \frac{\eta_N(\tau)}{\tau} \sum_{m\neq 0}\sum_{r(m)} e^{-i\tau m}e^{i\frac{\Lambda_m}{8\pi}\log 4 \tau}  \alpha_{j_1}\bar{\alpha}_{j_2}\alpha_{j_3} \, d\tau.
\end{equation*}

The goal is to show that the map $ \Tau_{\eta_N} $ is a contraction.

\begin{Lemma}
\label{Lp:est} 
Let  $ 1 < p < + \infty $, let $ s > \max\{ 1/p, 1-1/p\} $ and let $ N \geq 0 $ be a natural number. Then the $ \hat{L}^p $ estimates are as follows: for $ \nu \geq N $
\begin{equation*}
\| \Tau_{\eta_N,0}(R)\|_{l^p(\mathbb{Z};\hat{L}^p(I_{\nu}))} \leq \frac{C}{\nu+1}   \|\alpha \|^3_{l^p(\mathbb{Z})}
\end{equation*}
\begin{align*}
\|  \Tau_{\eta_N,1}(R)\|_{l^p(\mathbb{Z};\hat{L}^P(I_{\nu}))} \leq  \frac{C}{\nu+1} \max\left\{ \|\alpha \|^2_{l^{\infty}(\mathbb{Z})}, \frac{\|\alpha \|^2_{l^p(\mathbb{Z})}}{\nu+1} \right\} \|R\|_{X^{s}_p} 
\end{align*}
%
%
and

\begin{equation*}
\| \Tau_{\eta_N,2}(R) \|_{l^p(\mathbb{Z};\hat{L}^p(I_{\nu})} \leq \frac{C}{(\nu+1)^{2}}  \max\left\{ \|\alpha \|_{l^p(\mathbb{Z})}, \frac{    \|R\|_{X^{s}_p}     }{\nu+1} \right\}  \|R\|_{X^{s}_p}^2.
\end{equation*}

\end{Lemma}

\begin{Lemma}
\label{Beta:est}

Let  $ 1 < p < + \infty $, let $ s \in (\max\{ 1/p, 1-1/p\} , 1) $ and let $ N \geq 0 $ be a natural number. Then  the $\dot{\mathcal{H}}^{s}_{p} $ seminorm can be estimated for $ \nu \geq N $ as 
\begin{equation*}
 \| \Tau_{\eta_N,0}(R)\|_{l^p(\mathbb{Z};\dot{\mathcal{H}}^{s}_{p}(I_{\nu}))} \leq \frac{C}{\nu+1}   \| 	\alpha  \|^3_{l^p(\mathbb{Z})}
\end{equation*}
\begin{align*}
\|\Tau_{\eta_N,1}(R)\|_{l^p(\mathbb{Z};\dot{\mathcal{H}}^{s}_{p}(I_{\nu}))} \leq  \frac{C}{(\nu+1)} \max\left\{ \|\alpha \|^2_{l^{\infty}(\mathbb{Z})},  \frac{\|\alpha \|^2_{l^p(\mathbb{Z})}}{\nu+1} \right\} \|R\|_{X^{s}_p}
\end{align*}
and
\begin{equation*}
\| \Tau_{\eta_N,2}(R) \|_{l^p(\mathbb{Z};\dot{\mathcal{H}}^{s}_{p}(I_{\nu})} \leq \frac{C}{(\nu+1)^{2}}  \max\left\{ \|\alpha \|_{l^p(\mathbb{Z})}, \frac{    \|R\|_{X^{s}_p}      }{\nu+1} \right\}  \|R\|_{X^{s}_p}^2.
\end{equation*}

\end{Lemma}

\begin{proof}[Proof of Theorem \ref{Theo:p=2}]

The difficult parts of the proof are the estimates in the above lemmas. From Lemma \ref{Lp:est} and \ref{Beta:est}, together with the fact that $ \Tau_{\eta_N} $ is identically zero in $ I_{\nu} $ for $ \nu \leq N-1 $, we deduce that  
\begin{align*}
\sup_{\nu \geq 0 } \| (\nu+1) \Tau_{\eta_N}(R)\|_{l^p(\mathbb{Z}; \mathcal{H}^{s}_{p}(I_{\nu}))} \leq  C \sup_{\nu \geq N }\Bigg( \|\alpha \|^3_{l^p(\mathbb{Z})} + & \|\alpha \|^2_{l^{\infty}(\mathbb{Z})} \|R\|_{X^{s}_p} + \frac{\|\alpha \|^2_{l^p(\mathbb{Z})}}{\nu+1}\|R\|_{X^{s}_p} \\ & \, + \frac{\|\alpha \|_{l^p(\mathbb{Z})}}{(\nu+1)}\|R\|_{X^{s}_p}^2  + \frac{1}{(\nu+1)^{2}} \|R\|_{X^{s}_p}^3 \Bigg),  
\end{align*}
which can be shortly rewritten as
\begin{equation*}
\|\Tau_{\eta_N}(R)\|_{X^{s}_p} \leq C  \left( \|\alpha \|^3_{l^p(\mathbb{Z})} + \|\alpha \|^2_{l^{\infty}} \| R \|_{X^{ s}_p} +\frac{\|\alpha \|^2_{l^p}}{N+1} \| R \|_{X^{ s}_p} + \frac{\|\alpha \|_{l^p}}{(N+1)} \| R\|_{X^{ s}_p}^2 + \frac{1}{(N+1)^{2}} \| R\|_{X^{ s}_p}^3   \right).
\end{equation*}
for any natural number $ N \geq 0 $. Similarly one has 
\begin{equation*}
\|\Tau_{\eta_N}(R) - \Tau_{\eta_N}(Q)\|_{X^{s}_p} \leq C  \left(\|\alpha \|^2_{l^{\infty}}  +\frac{\|\alpha \|^2_{l^p}}{N+1} + \frac{\|\alpha \|_{l^p}}{(N+1)} (\| R\|_{X^{ s}_p} + \| Q\|_{X^{ s}_p}) + \frac{1}{(N+1)^{2}} \left(\| R\|_{X^{ s}_p}^2+\| Q\|_{X^{ s}_p}^2\right)   \right)\| R - Q\|_{X^{ s}_p}.
\end{equation*}
To show well-posedness is enough to define a complete metric space where the map $ \Tau_{\eta_N} $ is a contraction. 

\begin{itemize}

\item (Large in time well-posedness for small datum) To show large in time well-posedness for small initial datum we choose $ N = 0 $ and we define the space 
$ \mathcal{C}^{Lt} = \{ R \in X^{ s}_p $ such that $  \| R \|_{X^{ s}_p} \leq \bar{\varepsilon}  \}$. The functional 
$$ \Tau_{\eta_0} : \mathcal{C}^{Lt} \longrightarrow \mathcal{C}^{Lt} $$
is a contraction for $ \{ \alpha \}_{l^p} $ and $ \bar{\eps} $  small enough.

\item (Short in time well-posedness) To show short in time well-posedness we define the space 
$ \mathcal{C}^{St} = \{ R \in X^{ s}_p $ such that $  \| R \|_{X^{ s}_p} \leq 2 C\|\alpha\|_{l^p}^3 \}$. For $ \{ \alpha \}_{l^\infty} $  small enough, there exists $ N $ sufficiently large such that the functional 
$$ \Tau_{\eta_N} : \mathcal{C}^{St} \longrightarrow \mathcal{C}^{St} $$
is a contraction. 

\end{itemize} 

Note that in the large in time well-posedness result $ R $ satisfy the equation only for $ t \geq \pi $. But by choosing from the beginning $ \eta(x) = 1 $ for $ x \geq \delta > 0 $,  we deduce the existence of solution in the interval $[\delta, + \infty) $ for any fixed $ \delta $, due to the trivial bound $ 1/t \leq \delta^{-1 } $.

Finally uniqueness follows from the fact that solutions are obtained by a Picard iteration process. In particular if $ \eta $ and $ \bar{\eta} $ are two cut-off such that $ \eta(t) = \bar{\eta}(t) = 1 $ for $ t \geq t_{*} $ then $ n $-th iterate associated with $ \eta $ and $ \bar{\eta}$ are equal for time $ t \geq t_{*} $.

\end{proof}

Finally let us remark that the smallness in $ l^{\infty} $ of the data in the short in time result is due to the term $ C \| \alpha|^2_{l^{\infty}} \| R \|_{X^{ s}_p} $.

\subsection{Proof of Lemma \ref{Lp:est} and \ref{Beta:est}}

In this subsection we prove the two main lemmas. Let start with Lemma \ref{Lp:est}.

\subsubsection{Proof of Lemma \ref{Lp:est}}

\begin{proof}

We divide the proof in three parts where we tackle the terms  $ \Tau_0 $,  $ \Tau_1 $ and  $ \Tau_2 $ separately. In particular we show all the details for the estimates of  $ \Tau_0 $ and  $ \Tau_1 $. In the last part we explain how to deduce the estimates for $ \Tau_2 $ from the one of $ \Tau_1 $.

\item

\paragraph*{Estimates of $ \Tau_0 $}

Recalling the definition of $ \Tau_0 $ and by an integration by parts we deduce

\begin{align}
(\Tau_0(R))_k = & \, \frac{i \eta_N(t)}{8 \pi} \int_{t}^{\infty} \frac{1}{\tau} \sum_{m\neq 0}\sum_{r(m)} \eta_{N}(\tau) e^{-i\tau m} e^{i\frac{\Lambda_m}{8\pi}\log 4 \tau} \alpha_{j_1}\bar{\alpha}_{j_2}\alpha_{j_3} \, d\tau  \nonumber \\
= & \,  -\frac{i \eta^2_N(t)}{8 \pi}  \sum_{m\neq 0}\sum_{r(m)} \frac{e^{-i t m}}{i t m} e^{i\frac{\Lambda_m}{8\pi}\log 4 t} \alpha_{j_1}\bar{\alpha}_{j_2}\alpha_{j_3} \label{dec:F0}  +  \frac{i \eta_N(t)}{8 \pi} \int_{t}^{\infty} \sum_{m\neq 0}\sum_{r(m)} \eta'_{N}(\tau) \frac{e^{-i\tau m}}{i m \tau} e^{i\frac{\Lambda_m}{8\pi}\log 4 \tau}\alpha_{j_1}\bar{\alpha}_{j_2}\alpha_{j_3} \, d\tau   \\ & \, -  \frac{i \eta_N(t)}{8 \pi}  \int_{t}^{\infty} \sum_{m\neq 0}\sum_{r(m)} \frac{e^{-i\tau m}}{i m \tau^2 } \left(\eta_N(\tau) - \frac{i \eta_N(\tau) \Lambda_m}{8 \pi} \right)e^{i\frac{\Lambda_m}{8\pi}\log 4 \tau}\alpha_{j_1}\bar{\alpha}_{j_2}\alpha_{j_3} \, d\tau.  \nonumber 
\end{align}
In the sequel it is useful to decompose the integral on $ (t,+\infty) $ as the sum of integrals on $ I_{\mu} $. To do that let us introduce a partition of unity $ \left\{\psi_{\mu} \right\} $ on $ [\pi, + \infty) $  defined as $ \psi_{\mu}(t) = \eta_{\mu}(t) - \eta_{\mu+1}(t) $. It holds that $ 0 \leq \psi_{\eta} \leq 1 $, $ \psi_{\mu}$ is smooth and supported in $ I_{\mu} = [\pi \mu, \pi(\mu +2)] $. Moreover let us denote by
$$ \psi^e_{\nu}  = \sum_{d = -1}^1 \psi_{\nu+d} = \eta_{\nu-1} - \eta_{\nu+1} = \eta_{\nu-1}(1-\eta_{\nu+1}). $$ 
Observe that $ \psi_{\nu}^e$ is  smooth, with uniformly bounded derivative, supported in $ I^e_{
\nu} $, identically $ 1 $ in $ I_{\nu} $ and $ 0 \leq \psi_{\nu}^e \leq 1 $ .

Let us recall that the $ \hat{L}^p $ norm in $ I_{\nu} $ is define as the infimum of the  $ \hat{L}^p $ norm in $ I^e_{\nu} $ of all the possible extensions. In particular for $ \nu > N +1  $, we use the extension $ (\tilde{\Tau}_{0}(R))_k = \psi^{e}_{\nu}(\Tau_0(R))_k$. We deduce that
\begin{align*}
\| \Tau_0(R) \|_{l^p(\hat{L}^p(I_{\nu}))}^p \leq  \sum_{k} \sum_{h} \left|\left[\psi^{e}_{\nu}(\Tau_0(R))_k\right]^{\widehat{ }}_h\right|^p,
\end{align*}
where $ [.]^{\widehat{}}_h $ denotes the $h$-th Fourier coefficient. Let us compute $[\psi^{e}_{\nu}(\Tau_0(R))_k]^{\widehat{ }}_h$. First of all note that $ \eta_{N}(t) = 1  $ for $ t \geq (N+1)\pi $, in particular the second term of \eqref{dec:F0} is zero. Denote by 
\begin{equation*} 
\tilde{\psi}^e_{\nu}(t) = \frac{\nu+1}{8 \pi t} e^{i\frac{\Lambda_m}{8\pi}\log 4 t} \psi^e_{\nu}, \quad \text{ which satisfies } \quad  \| \tilde{\psi}^e_{ \nu}\|_{\tilde{H}^{s,p}} \leq C  (1 + \|\alpha \|_{l^{\infty}}^2),
\end{equation*}
with $ C $ independent of $ \nu \geq 1 $ and $ m $.  For the first term on the right hand side of \eqref{dec:F0}, we have
\begin{align*}
\frac{\psi^e_{\nu}(t)}{8 \pi}  \sum_{m\neq 0}\sum_{r(m)} \frac{e^{-i t m}}{i t m} e^{i\frac{\Lambda_m}{8\pi}\log 4 t} \alpha_{j_1}\bar{\alpha}_{j_2}\alpha_{j_3} = & \,  \frac{1}{\nu+1}\sum_{m\neq 0}\sum_{r(m)} \sum_{l} \frac{e^{-i t (m-l/2)}}{i m} [\tilde{\psi}^e_{\nu}]^{\widehat{}}_{l} \alpha_{j_1}\bar{\alpha}_{j_2}\alpha_{j_3}  \\ = & \,  \frac{1}{\nu+1} \sum_{h} e^{-ith/2}\left(\sum_{\substack{l \in \mathbb{Z} \\ l+h \text{ even} \\ l+h  \neq 0} }\sum_{r((h+l)/2)}  \frac{2}{i (l+h) }[\tilde{\psi}^e_{\nu}]^{\widehat{}}_{l}\alpha_{j_1}\bar{\alpha}_{j_2}\alpha_{j_3} \right).
\end{align*}
As we already mention the second term is zero for $\nu > N+1 $. For the last term in \eqref{dec:F0}, we integrate by parts to obtain, for $ h \neq 0 $
\begin{align*}
& \left| \int_{I_{\nu}^e} e^{-iht/2}\psi^{e}_{\nu}\int_{t}^{\infty} \sum_{m\neq 0}\sum_{r(m)} \frac{e^{-i\tau m}}{i m \tau^2 }\left(1 - \frac{\Lambda_m}{8 \pi} \right)e^{i\frac{\Lambda_m}{8\pi}\log 4 \tau}  \alpha_{j_1}\bar{\alpha}_{j_2}\alpha_{j_3} \, d\tau dt \right|  \\ & \, =   \vast| \int_{I_{\nu}^e} \frac{e^{-iht/2}}{ih/2}(\psi^{e}_{\nu})'\int_{t}^{\infty} \sum_{m\neq 0}\sum_{r(m)} \frac{e^{-i\tau m}}{i m \tau^2 }\left(1 - \frac{\Lambda_m}{8 \pi} \right)e^{i\frac{\Lambda_m}{8\pi}\log 4 \tau}  \alpha_{j_1}\bar{\alpha}_{j_2}\alpha_{j_3} \, d\tau dt \\ & \quad \quad + \int_{I_{\nu}^e} \frac{e^{-iht/2}}{ih/2} \psi^{e}_{\nu} \sum_{m\neq 0}\sum_{r(m)} \frac{e^{-i t m}}{i m t^2 } \left(1 - \frac{\Lambda_m}{8 \pi} \right)e^{i\frac{\Lambda_m}{8\pi}\log 4 t}  \alpha_{j_1}\bar{\alpha}_{j_2}\alpha_{j_3} dt\vast| \\
 & \, \lesssim   \frac{1}{ \langle h \rangle (\nu+1) } \sum_{m\neq 0}\sum_{r(m)} \frac{1}{|m|} \left|\alpha_{j_1}\bar{\alpha}_{j_2}\alpha_{j_3} \right|(1+\|\alpha\|_{l^{\infty}}^2).
\end{align*}
Then, taking $ 1/p + 1/q = 1 $ and recalling that the number of elements of $ r(m) $ are less or equal to $ C_{\eps} m^{\eps} $ for any $ \eps > 0 $, we have

\begin{align*}
\sum_{k}\|(\Tau^{0}(R))_{k} \|_{\hat{L}^p(I_{\nu})}^p \leq  & \, \frac{1}{(\nu+1)^p} \sum_{k} \sum_{h} \left| \sum_{\substack{l \in \mathbb{Z} \\ l+h \text{ even} \\ h+l \neq 0 }  }\sum_{r((h+l)/2)}  \frac{2}{i (l+h)\langle l \rangle } \langle l \rangle [\tilde{\psi}^e_{\nu}]^{\widehat{}}_{l} \alpha_{j_1}\bar{\alpha}_{j_2}\alpha_{j_3} \right|^p \\
& \, +  \frac{1}{(\nu+1)^p} \sum_{k}\sum_{h}  \frac{1}{ \langle h \rangle^p } \left|\sum_{m\neq 0}\sum_{r(m)} \frac{1}{|m|} \left|\alpha_{j_1}\bar{\alpha}_{j_2}\alpha_{j_3}\right|\right|^p \\
\leq & \, \frac{1}{(\nu+1)^p} \sum_{k} \sum_{h} \left( \sum_{\substack{l \in \mathbb{Z} \\ l+h \text{ even} \\ h+l \neq 0} }\sum_{r((h+l)/2)}  \frac{2^q}{|l+h|^q\langle l \rangle^q } \right)^{\frac{p}{q}}  \sum_{\substack{l \in \mathbb{Z} \\ l+h \text{ even} \\ h+l \neq 0} }\sum_{r((h+l)/2)} \langle l \rangle^p \left( [\tilde{\psi}^e_{\nu}]^{\widehat{}}_{l}\right)^p |\alpha_{j_1}\bar{\alpha}_{j_2}\alpha_{j_3}|^p  \\
& \, +  \frac{1}{(\nu+1)^p} \sum_{k}\left(\sum_{h}  \frac{1}{ \langle h \rangle^p }\right) \sum_{m\neq 0}\sum_{r(m)} \left|\alpha_{j_1}\bar{\alpha}_{j_2}\alpha_{j_3}\right|^p\left( \sum_{m \neq 0}\sum_{r(m)} \frac{1}{|m|^{q}}\right)^{\frac{p}{q}} \\
\lesssim & \, \frac{1}{(\nu+1)^p} \sum_{k} \sum_{h} \frac{1}{\langle h \rangle^{p-\eps}} \sum_{\substack{l \in \mathbb{Z} \\ l+h \text{ even} \\ h+l \neq 0 } }\sum_{r((h+l)/2)} \langle l \rangle^p \left([\tilde{\psi}^e_{\nu}]^{\widehat{}}_{l}\right)^p |\alpha_{j_1}\bar{\alpha}_{j_2}\alpha_{j_3}|^p + \frac{1}{\nu^p} \|\alpha \|_{l^p}^3 \left( \sum_{m \neq 0} \frac{1}{|m|^{q-\eps}}\right)^{\frac{p}{q}} \\
\lesssim & \, \frac{1}{(\nu+1)^p} \left(\sum_{k} \sum_{h}  \sum_{\substack{l \in \mathbb{Z} \\ l+h \text{ even} \\ h+l \neq 0} }\sum_{r((h+l)/2)} \langle l \rangle^p \left([\tilde{\psi}^e_{\nu}]^{\widehat{}}_{l}\right)^p |\alpha_{j_1}\bar{\alpha}_{j_2}\alpha_{j_3}|^p\right)\sup_{h}\frac{1}{\langle h \rangle^{p-\eps}} + \frac{1}{\nu^p} \|\alpha \|_{l^p}^3 \\ 
\lesssim & \,  \frac{1}{(\nu+1)^p} \left(\sum_{k} \sum_{m \neq 0}  \sum_{l}\sum_{r(m)} \langle l \rangle^p \left([\tilde{\psi}^e_{\nu}]^{\widehat{}}_{l}\right)^p |\alpha_{j_1}\bar{\alpha}_{j_2}\alpha_{j_3}|^p\right) + \frac{1}{\nu^p} \|\alpha \|_{l^p}^3 \\ 
\lesssim & \,  \frac{1}{(\nu+1)^p} \left( \sum_{l} \langle l \rangle^p \left([\tilde{\psi}^e_{\nu}]^{\widehat{}}_{l}\right)^p  \right) \left(\sum_{k} \sum_{m\neq 0}  \sum_{r(m)}  |\alpha_{j_1}\bar{\alpha}_{j_2}\alpha_{j_3}|^p\right) + \frac{1}{\nu^p} \|\alpha\|_{l^p}^3  \\
\lesssim & \, \frac{1}{(\nu+1)^p} \|\alpha\|_{l^p}^3,
\end{align*}  
where we choose $ \eps < q - 1 $.

Let us consider the case $ \nu \in \{ N, N+1 \} $. Notice that $ \Tau_0(R)(t) $ is zero for $ t \leq \pi N $. Then it is enough to multiply $ \Tau_0 $ by $ (1-\eta_{\nu+1}) $ to extend $ \Tau_0 $ to a periodic function in $ I_{\nu}^e $ for $ \nu \in \{ N, N+1 \} $.
While estimating the $ l^p(\hat{L}^p)$ norm of $ \Tau_0 $ the only new term is 
\begin{equation*}
\eta_N(t)(1-\eta_{\nu+1}) \int_{t}^{\infty} \sum_{m\neq 0}\sum_{r(m)} \eta'_{N}(\tau) \frac{e^{-i\tau m}}{i m \tau} e^{i\frac{\Lambda_m}{8\pi}\log 4 \tau} \alpha_{j_1}\bar{\alpha}_{j_2}\alpha_{j_3} \, d\tau.
\end{equation*}
Notice that the $0$-th Fourier coefficient is bounded by 
\begin{equation*}
\frac{1}{(\nu+1) } \sum_{m\neq 0}\sum_{r(m)} \frac{1}{|m|} \left|\alpha_{j_1}\bar{\alpha}_{j_2}\alpha_{j_3}\right|.
\end{equation*}
Let $ h \neq 0 $, then the $h$-th Fourier coefficient  
\begin{align*}
& \left| \int_{I_{\nu}^e} e^{-iht/2}\eta_N(t)(1-\eta_{\nu+1}) \int_{t}^{\infty} \sum_{m\neq 0}\sum_{r(m)} \eta'_{N}(\tau) \frac{e^{-i\tau m}}{i m \tau} e^{i\frac{\Lambda_m}{8\pi}\log 4 \tau} \alpha_{j_1}\bar{\alpha}_{j_2}\alpha_{j_3} \, d\tau dt \right|  \\ & \, =   \vast| \int_{I_{\nu}^e} \frac{e^{-iht/2}}{ih/2}\partial_t(\eta_N(t)(1-\eta_{\nu+1})) \int_{t}^{\infty} \sum_{m\neq 0}\sum_{r(m)} \eta'_{N}(\tau) \frac{e^{-i\tau m}}{i m \tau} e^{i\frac{\Lambda_m}{8\pi}\log 4 \tau} \alpha_{j_1}\bar{\alpha}_{j_2}\alpha_{j_3} \, d\tau dt \\ & \quad \quad + \int_{I_{\nu}^e} \frac{e^{-iht/2}}{ih/2} \eta_N(t)(1-\eta_{\nu+1}) \sum_{m\neq 0}\sum_{r(m)} \eta'_{N}(t) \frac{e^{-i t m}}{i m t} e^{i\frac{\Lambda_m}{8\pi}\log 4 t} \alpha_{j_1}\bar{\alpha}_{j_2}\alpha_{j_3}  dt\vast| \\
& \, \lesssim   \frac{1}{ \langle h \rangle (\nu+1) } \sum_{m\neq 0}\sum_{r(m)} \frac{1}{|m|} \left|\alpha_{j_1}\bar{\alpha}_{j_2}\alpha_{j_3}\right|,
\end{align*}
which can be treated as before.
This concludes the proof for the term $ \Tau_0 $ because we can argue as for $ \nu > N+1 $.

\paragraph*{Estimates of $ \Tau_1 $} 

The linear term $ \Tau_1 $ is the sum of five terms, more precisely it reads
\begin{align}
(\Tau_1(R))_k = & \, \frac{i \eta_N(t)}{8 \pi} \int_{t}^{\infty} \frac{\eta_{N}(\tau)}{\tau}\sum_{m}\sum_{r(m)}  e^{-i\tau m} e^{i\frac{\Lambda_m}{8\pi}\log 4 \tau} R_{j_1}\bar{\alpha}_{j_2}\alpha_{j_3} \, d\tau + \frac{i \eta_N(t)}{8 \pi} \int_{t}^{\infty} \frac{\eta_{N}(\tau)}{\tau}\sum_{m}\sum_{r(m)}  e^{-i\tau m} e^{i\frac{\Lambda_m}{8\pi}\log 4 \tau} \alpha_{j_1}\bar{R}_{j_2}\alpha_{j_3} \, d\tau  \nonumber \\ &  \, + \frac{i \eta_N(t)}{8 \pi} \int_{t}^{\infty} \frac{\eta_{N}(\tau)}{\tau}\sum_{m}\sum_{r(m)}  e^{-i\tau m} e^{i\frac{\Lambda_m}{8\pi}\log 4 \tau} \alpha_{j_1}\bar{\alpha}_{j_2}R_{j_3} \, d\tau   - \frac{i \eta_N(t)}{8 \pi} \int_{t}^{\infty} \frac{\eta_{N}(\tau)}{\tau} \alpha_k \bar{R}_k \alpha_k \, d \tau \label{MAMO} \\ & \, - \frac{i \eta_N(t)}{8 \pi} \int_{t}^{\infty} \frac{\eta_{N}(\tau)}{\tau} \alpha_k R_k \bar{\alpha}_k \, d \tau.	 \nonumber
\end{align}	 
Let us denote the first term of the right hand side by $ (F(R))_k $. In the following we estimates only the term $ F(R) $. The estimate of the other terms follow similarly. As in the previous case, we start by considering $ \nu > N+1 $. We have
\begin{equation*}
\| F(R)\|_{l^p(\mathbb{Z};\hat{L}^p(I_{\nu}))}^p \leq \sum_{k}\sum_{h} \left|[\psi^e_{\nu}(F(R))_k]^{\widehat{}}_h\right|^p.
\end{equation*}
First of all, we note that $ \psi_{\nu}^e (F(R))_k $ is supported in the interval $ I_{\nu}^{e} $, so we can decompose $ (F(R))_k $ as the sum of a time dependent part plus a constant as follows   
\begin{align}
(F(R))_k = & \, \frac{i}{8 \pi}\int_{t}^{\infty} \frac{1}{\tau}\sum_{m}\sum_{r(m)} e^{-i\tau m} e^{i\frac{\Lambda_m}{8\pi}\log 4 \tau} R_{j_1}\bar{\alpha}_{j_2}\alpha_{j_3} \, d\tau \nonumber  \\
= & \,\frac{i}{8 \pi} \int_{t}^{\infty} \sum_{\mu \geq \nu -2}\frac{\psi_{\mu}}{\tau}\sum_{m}\sum_{r(m)} e^{-i\tau m}e^{i\frac{\Lambda_m}{8\pi}\log 4 \tau} R_{j_1}\bar{\alpha}_{j_2}\alpha_{j_3} \, d\tau \nonumber \\
= & \, \frac{i}{8 \pi} \sum_{\nu-2 \leq \mu \leq \nu + 2}  \int_{t}^{\infty} \frac{\psi_{\mu}}{\tau}\sum_{m}\sum_{r(m)} e^{-i\tau m} e^{i\frac{\Lambda_m}{8\pi}\log 4 \tau} R_{j_1}\bar{\alpha}_{j_2}\alpha_{j_3} \, d\tau \nonumber \\ & \, + \frac{i}{8 \pi} \sum_{ \mu \geq \nu + 3}  \int_{I_{\mu}} \frac{\psi_{\mu}}{\tau}\sum_{m}\sum_{r(m)} e^{-i\tau m} e^{i\frac{\Lambda_m}{8\pi}\log 4 \tau} R_{j_1}\bar{\alpha}_{j_2}\alpha_{j_3} \, d\tau. \label{dec:F1}
\end{align}
Let us start with the constant term. Note that 
\begin{align*}
\frac{1}{8 \pi} \int_{I_{\mu}} \frac{\psi_{\mu}}{\tau}\sum_{m \neq 0}\sum_{r(m)} e^{-i\tau m} e^{i\frac{\Lambda_m}{8\pi}\log 4 \tau} R_{j_1}\bar{\alpha}_{j_2}\alpha_{j_3} = & \, \frac{1}{\mu+1} \sum_{m \neq 0}\sum_{r(m)} \left( \int_{I_{\mu}} e^{-i\tau m} \tilde{\psi}_{\mu} R_{j_1} \, d\tau \right) \bar{\alpha}_{j_2}\alpha_{j_3} \\
= & \,  \frac{1}{\mu+1} \sum_{m \neq 0}\sum_{r(m)}  [\tilde{\psi}_{\mu} R_{j_1}]^{\widehat{ }}_m \bar{\alpha}_{j_2}\alpha_{j_3},  
\end{align*}
where
\begin{equation*} 
\tilde{\psi}_{\mu}(t) = \frac{\mu+1}{8 \pi t} e^{i\frac{\Lambda_m}{8\pi}\log 4 t} \psi_{\mu}, \quad \text{ which satisfies } \quad  \| \tilde{\psi}_{ \mu}\|_{\tilde{H}^{s,p}} \leq C  (1 + \|\alpha\|_{l^{\infty}}^2),
\end{equation*}
with $ C $ independent of $ \mu \geq 0 $ and $ m $.

By using the previous computation we have
\begin{align*}
  \sum_k \vast| \sum_{ \mu \geq \nu + 3} & \int_{I_{\mu}} \frac{\psi_{\mu}}{8 \pi \tau}\sum_{m \neq 0}\sum_{r(m)} e^{-i\tau m} 
  e^{i\frac{\Lambda_m}{8\pi}\log 4 \tau}
  R_{j_1}\bar{\alpha}_{j_2}\alpha_{j_3} \vast|^p \lesssim \sum_{k}  \left| \sum_{ \mu \geq \nu + 3} \frac{1}{\mu+1} \sum_{m \neq 0}\sum_{r(m)}  [\tilde{\psi}_{\mu} R_{j_1}]^{\widehat{ }}_m \bar{\alpha}_{j_2}\alpha_{j_3} \right|^p \\
\lesssim & \, \sum_{k} \left| \sum_{ \mu \geq \nu + 3} \frac{1}{(\mu+1)^{2}} \sum_{m \neq 0}\sum_{r(m)}  \frac{\langle m \rangle^{s}}{\langle m \rangle^{s}} \mu [\tilde{\psi}_{\mu} R_{j_1}]^{\widehat{ }}_m  \bar{\alpha}_{j_2}\alpha_{j_3} \right|^p \\
\lesssim & \, \sum_{k} \left| \sum_{ \mu \geq \nu + 3} \frac{1}{(\mu+1)^{2}} \sum_{m \neq 0}\sum_{r(m)}  \frac{\langle m \rangle^{s}}{\langle m \rangle^{s}} \mu [\tilde{\psi}_{\mu} R_{j_1}]^{\widehat{ }}_m  \right|^p \|\alpha \|_{l^{\infty}}^{2p} \\
\lesssim & \, \left| \sum_{ \mu \geq \nu + 3}  \frac{1}{(\mu+1)^{2}} \left| \sum_{k}  \left|\sum_{m \neq 0}\sum_{r(m)}  \frac{\langle m \rangle^{s}}{\langle m \rangle^{s}} \mu [\tilde{\psi}_{\mu} R_{j_1}]^{\widehat{ }}_m \right|^p \right|^{\frac{1}{p}} \right|^p \|\alpha \|_{l^{\infty}}^{2p} \\
\lesssim & \, \frac{1}{(\nu+1)^{p}}\sup_{\mu}\left\{ \sum_{k} \sum_{m\neq 0} \sum_{r(m)} \left|\frac{\langle m \rangle^{s}}{\langle m \rangle^{\eps}} \mu [\tilde{\psi}_{\mu} R_{j_1}]^{\widehat{ }}_m  \right|^p \left( \sum_{m \neq 0}\sum_{r(m)} \frac{1}{\langle m \rangle^{(s-{\eps})q}}\right)^{\frac{p}{q}} \right\} \|\alpha \|_{l^{\infty}}^{2p} \\
\lesssim & \, \frac{1}{(\nu+1)^{ p}}\sup_{\mu}\left\{  \sum_{m\neq 0} \sum_{r(m)} \sum_{k}\left|\frac{\langle m \rangle^{s}}{\langle m \rangle^{\eps}} \mu [\tilde{\psi}_{\mu} R_{j_1}]^{\widehat{ }}_m  \right|^p  \right\} \|\alpha\|_{l^{\infty}}^{2p} \\
\lesssim & \, \frac{1}{(\nu+1)^{ p}}\sup_{\mu}\left\{  \sum_{m\neq 0} \sum_{k}\left| \langle m \rangle^{s} \mu [\tilde{\psi}_{\mu} R_{j_1}]^{\widehat{ }}_m  \right|^p  \right\} \|\alpha \|_{l^{\infty}}^{2p} \\
\lesssim & \, \frac{\|\alpha \|_{l^{\infty}}^{2p}}{(\nu+1)^{ p}} \sup_{\mu} \|\mu R\|_{l^p(\mathbb{Z}; \mathcal{H}^{s}_p(I_{\mu}))}^p,
\end{align*} 
where we recall that $ j_1 = k-z $ for $ z \in r(m) $. Moreover in the last inequality we use the fact that for any function $ f \in \mathcal{H}^s_p(I_{\mu}) $ and $ \psi \in C^{\infty}_c((\mu \pi, (\mu +2 ) \pi )) $ the function $ f \psi (t) = f(t)\psi(t) $ in $ I_{\mu} $ and $ f \psi (t) = 0 $ in $ I_{\mu}^e \setminus I_{\mu} $ satisfies  
\begin{equation}
\label{ine:inf}
\|\psi f  \|_{\tilde{H}^{s,p}(I_{\mu}^e)} \leq C\| \psi \|_{\tilde{H}^{s,p}(I_{\mu}^e)} \inf_{\substack{\tilde{f}\in \tilde{H}^{s,p}(I^{e}_{\nu}),  \\ \tilde{f}|_{I_{\nu}}  = f}}  \| \tilde{f} \|_{\tilde{H}^{s,p}(I_{\nu}^e)} = C\| \psi \|_{\tilde{H}^{s,p}(I_{\mu}^e)} \| f \|_{\mathcal{H}^s_p(I_{\mu})}, 
\end{equation}
for $ s > 1-1/p $. Inequality \eqref{ine:inf} holds true because for any $ \tilde{f}\in \tilde{H}^{s,p}(I^{e}_{\nu}) $ such that $ \tilde{f}|_{I_{\nu}}  = f $, it holds $ \tilde{f} \psi = f \psi $. Moreover using the fact that $ \tilde{H}^{s,p} $ is an algebra for $ s > 1-1/p $, we have
\begin{equation*}
\|\psi f  \|_{\tilde{H}^{s,p}(I_{\mu}^e)} = \|\psi \tilde{f}  \|_{\tilde{H}^{s,p}(I_{\mu}^e)} \leq C\|\psi  \|_{\tilde{H}^{s,p}(I_{\mu}^e)}\|\tilde{f}  \|_{\tilde{H}^{s,p}(I_{\mu}^e)}.
\end{equation*} 
Taking the infimum on both sides, we deduce \eqref{ine:inf}. 

The estimate for the second term of \eqref{dec:F1} follows then from the fact that $ \psi^e_{\nu} $ is bounded in $ \hat{L}^p(I_{\nu}^e)$ independently of $ \nu $. 

Let us now treat the remaining terms of \eqref{dec:F1}. In other words the once with $ \mu \in [\nu-2,\nu+2] $. Without loss of generality we consider only the case $ \mu = \nu $. Let us compute the $ h $-th Fourier coefficient
\begin{align*}
\int_{I_{\nu}^e} e^{-i \frac{h}{2}t} \psi^e_{\nu}(t)& \int_{t}^{\pi(\nu+2)} \frac{\psi_{\nu}(\tau)}{8 \pi \tau} \sum_{m\neq 0} \sum_{r(m)} e^{-i m\tau} e^{i\frac{\Lambda_m}{8\pi}\log 4 \tau} R_{j_1} \bar{\alpha}_{j_2} \alpha_{j_3} \, d\tau \, dt \\ 
= & \, \int_{I_{\nu}^e} e^{-i \frac{h}{2}t} \psi^e_{\nu}(t)\int_{t}^{\pi(\nu+2)} \frac{1}{(\nu+1)^{2}} \sum_{m\neq 0} \sum_{r(m)} \sum_{l} e^{-i (m-l) \tau} [(\nu+1) \tilde{\psi_{\nu}} R_{j_1}]^{\widehat{ }}_l \bar{\alpha}_{j_2} \alpha_{j_3} \, d\tau \, dt  \\
= & \,  \frac{1}{(\nu+1)^{2}} \sum_{m\neq 0} \sum_{r(m)} \sum_{l} \left( \int_{I_{\nu}^e} e^{-i \frac{h}{2}t} \psi^e_{\nu}(t)\int_{t}^{\pi(\nu+2)}  e^{-i (m-l) \tau} \, d\tau \, dt  \right) [(\nu+1) \tilde{\psi_{\nu}}  R_{j_1}]^{\widehat{ }}_l \bar{\alpha}_{j_2} \alpha_{j_3}.
\end{align*}   
For $ h \neq 0$ and by two integration by parts we have 
\begin{align*}
\left|\int_{I_{\nu}^e} e^{-i \frac{h}{2}t} \psi^e_{\nu}(t)\int_{t}^{\pi(\nu+2)}  e^{-i (m-l) \tau} \, d\tau \, dt \right| \lesssim & \, \left|\int_{I_{\nu}^e} \frac{2}{h } e^{-i \frac{h}{2}t} (\psi^e_{\nu})'(t)\int_{t}^{\pi(\nu+2)}  e^{-i (m-l) \tau} \, d\tau \, dt \right| \\
& \, + \left|\int_{I_{\nu}^e} \frac{2}{h}e^{-i \frac{h}{2}t} \psi^e_{\nu}(t)  e^{-i (m-l) t}  dt \right| \\
\leq & \, \frac{1}{h} \frac{1}{\langle m-l \rangle} + \frac{1}{h} \frac{1}{\langle m - l + h/2 \rangle} . 
\end{align*}
We deduce that 
\begin{align*}
\sum_{k} \sum_{h} & \left|\left[\int_{t}^{\pi(\nu+2)} \frac{\psi_{\nu}}{8 \pi \tau}\sum_{m}\sum_{r(m)} e^{-i\tau m} e^{i\frac{\Lambda_m}{8\pi}\log 4 \tau} R_{j_1}\bar{\alpha}_{j_2} \alpha_{j_3} \, d\tau \right]^{\widehat{ }}_h \right|^p  \\
\lesssim & \, \frac{1}{(\nu+1)^{(2)p}}\sum_{k} \sum_{h}\left( \sum_{m}\sum_{r(m)}\sum_{l} \frac{1}{\langle h \rangle} \left( \frac{1}{\langle m-l \rangle} + \frac{1}{\langle m - l + h/2 \rangle} \right)\frac{\langle l \rangle^s }{\langle l \rangle^s} [(\nu+1) \tilde{\psi}_{\nu} R_{j_1}]^{\widehat{ }}_l \bar{\alpha}_{j_2} \alpha_{j_3} \right)^p \\
\lesssim & \, \frac{\|\alpha\|_{l^p}^{2p}}{(\nu+1)^{(2)p}}\sup_{\mu} \|(\mu +1) R\|_{l^p(\mathbb{Z}; \mathcal{H}^{s}_p(I_{\mu}))}^p.  
\end{align*}
Note that in the last estimate we use H\"older's inequality and the restriction to $ qs > 1 $ i.e. $ s > 1- 1/p $ and \eqref{ine:inf}.

Let us now consider the case $ \nu \in \{N, N+1 \}$. As before, we extend $ F(R) $ by $(1-\eta_{\nu+1}) F(R) $. Moreover we decompose $ F_k(R) $ as the sum of a time dependent part plus a constant as follows   
\begin{align*}
(F(R))_k = & \, \frac{i\eta_{N}(t)}{8 \pi}\int_{t}^{\infty} \frac{\eta_{N}(\tau)}{\tau}\sum_{m}\sum_{r(m)} e^{-i\tau m} e^{i\frac{\Lambda_m}{8\pi}\log 4 \tau} R_{j_1}\bar{\alpha}_{j_2} \alpha_{j_3} \, d\tau \nonumber  \\
= & \,\frac{i\eta_{N}(t)}{8 \pi} \int_{t}^{\infty} \sum_{\mu \geq N }\frac{\psi_{\mu}}{\tau}\sum_{m}\sum_{r(m)} e^{-i\tau m}e^{i\frac{\Lambda_m}{8\pi}\log 4 \tau} R_{j_1}\bar{\alpha}_{j_2} \alpha_{j_3} \, d\tau \nonumber \\
= & \, \frac{i\eta_{N}(t)}{8 \pi} \sum_{N \leq \mu \leq \nu +2 }  \int_{t}^{\infty} \frac{\psi_{\mu}}{\tau}\sum_{m}\sum_{r(m)} e^{-i\tau m} e^{i\frac{\Lambda_m}{8\pi}\log 4 \tau} R_{j_1}\bar{\alpha}_{j_2} \alpha_{j_3} \, d\tau \\ & \, + \frac{i\eta_{N}(t)}{8 \pi} \sum_{ \mu \geq \nu + 3}  \int_{I_{\mu}} \frac{\psi_{\mu}}{\tau}\sum_{m}\sum_{r(m)} e^{-i\tau m} e^{i\frac{\Lambda_m}{8\pi}\log 4 \tau} R_{j_1}\bar{\alpha}_{j_2} \alpha_{j_3} \, d\tau.
\end{align*}
From this point we can argue as for the case $ \nu > N + 1 $. 

\paragraph*{Estimates of $ \Tau_2 $} The functional $ \Tau_2 $ is the sum of eight terms, six of them are quadratic in $ R $ and two are cubic. We will just look at the most delicate one. Namely 
\begin{equation*}
G_k(R) = \frac{i\eta_{N}(t)}{8 \pi} \int_{t}^{\infty} \frac{\eta_{N}(\tau)}{\tau}\sum_{m \neq 0}\sum_{r(m)} e^{-i\tau m} e^{i\frac{\Lambda_m}{8\pi}\log 4 \tau} R_{j_1}\bar{R}_{j_2}R_{j_3} \, d\tau.
\end{equation*} 
For $ \nu > N +1 $ we consider the extension $ \psi^e_{\nu}G(R) $ in $ I_{\nu}^e$. We have 
\begin{equation*}
\| G(R)\|_{l^p(\mathbb{Z};\hat{L}^p(I_{\nu}))}^p \leq \sum_{k}\sum_{h} \left|[\psi^e_{\nu}G(R)]^{\widehat{ }}_h\right|^p 
\end{equation*}
First of all we note that $ \psi_{\nu}^e (G(R))_k $ is supported in the interval $ I_{\nu}^{e} $, so we can decompose $ G(R) $ as the sum of a time dependent part plus a constant as follows  
\begin{align}
(G(R))_k = & \, \int_{t}^{\infty} \frac{1}{8 \pi \tau}\sum_{m \neq 0}\sum_{r(m)} e^{-i\tau m} e^{i\frac{\Lambda_m}{8\pi}\log 4 \tau} R_{j_1}\bar{R}_{j_2}R_{j_3} \, d\tau \nonumber  \\
= & \, \int_{t}^{\infty} \sum_{\mu \geq \nu -2}\frac{\psi_{\mu}}{8 \pi \tau}\sum_{m \neq 0}\sum_{r(m)} e^{-i\tau m} e^{i\frac{\Lambda_m}{8\pi}\log 4 \tau} R_{j_1}\bar{R}_{j_2}R_{j_3} \, d\tau \nonumber \\
= & \, \sum_{\nu-2 \leq \mu \leq \nu +2}  \int_{t}^{\infty} \frac{\psi_{\mu}}{8 \pi \tau}\sum_{m \neq 0}\sum_{r(m)} e^{-i\tau m} e^{i\frac{\Lambda_m}{8\pi}\log 4 \tau}  R_{j_1}\bar{R}_{j_2}R_{j_3} + \sum_{ \mu \geq \nu + 3}  \int_{I_{\mu}} \frac{\psi_{\mu}}{8\pi\tau}\sum_{m \neq 0}\sum_{r(m)} e^{-i\tau m} e^{i\frac{\Lambda_m}{8\pi}\log 4 \tau} R_{j_1}\bar{R}_{j_2}R_{j_3}. \label{dec:F7}
\end{align}
Let us start by considering the constant part. Note that 
\begin{align*}
\int_{I_{\mu}} \frac{\psi_{\mu}}{ 8 \pi \tau}\sum_{m \neq 0}\sum_{r(m)} e^{-i\tau m} e^{i\frac{\Lambda_m}{8\pi}\log 4 \tau} R_{j_1}\bar{R}_{j_2}R_{j_3} = & \, \frac{1}{\mu+1} \sum_{m \neq 0}\sum_{r(m)} \left( \int_{I_{\mu}} e^{-itm} \tilde{\psi}_{\mu} R_{j_1} \, d\tau  \bar{R}_{j_2} R_{j_3} \right) \\
= & \,  \frac{1}{(\mu+1)^{4}} \sum_{m \neq 0}\sum_{r(m)}  (\mu+1)^{3}[\tilde{\psi}_{\mu} R_{j_1} \bar{R}_{j_2} R_{j_3}]^{	\widehat{ }}_m.   
\end{align*}
To estimate the above term we follow exactly the case $ i = 1 $. For the time dependent term w.l.o.g. we consider only the case $ \mu = \nu $. Let compute the $ h $-th Fourier coefficient
\begin{align*}
\int_{I_{\nu}^e} e^{-i \frac{h}{2}t} \psi^e_{\nu}(t)& \int_{t}^{\pi(\nu+2)} \frac{\psi_{\nu}(\tau)}{8\pi \tau} \sum_{m\neq 0} \sum_{r(m)} e^{-i m\tau}e^{i\frac{\Lambda_m}{8\pi}\log 4 \tau} R_{j_1} \bar{R}_{j_2} R_{j_3} \, d\tau \, dt \\ 
= & \, \int_{I_{\nu}^e} e^{-i \frac{h}{2}t} \psi^e_{\nu}(t)\int_{t}^{\pi(\nu+2)} \frac{1}{(\nu+1)^{4}} \sum_{m\neq 0} \sum_{r(m)} \sum_{l} e^{-i (m-l) \tau} \widehat{(\nu+1)^{3}\tilde{\psi_{\nu}} R_{j_1} \bar{R}_{j_2} R_{j_3}}_l \, d\tau \, dt  \\
= & \,  \frac{1}{(\nu+1)^{4}} \sum_{m\neq 0} \sum_{r(m)} \sum_{l} \left( \int_{I_{\nu}^e} e^{-i \frac{h}{2}t} \psi^e_{\nu}(t)\int_{t}^{\pi(\nu+2)}  e^{-i (m-l) \tau} \, d\tau \, dt  \right) \widehat{(\nu+1) \tilde{\psi_{\nu}}  R_{j_1} \bar{R}_{j_2} R_{j_3}}_l.
\end{align*}   
From this point we can proceed as in the  case $ i = 1 $.

For $ \nu \in \{N, N+1 \}$, it is enough to extend $ G(R) $ by $ (1-\eta_{\nu+1}) G(R) $ and follow the estimates for $ \nu > N+1$.

\end{proof}

\subsubsection{Proof of Lemma \ref{Beta:est}}

\begin{proof}[Proof of Lemma \ref{Beta:est}]

As before we divide the proof in three parts. In particular we present all the details for the terms $ \Tau_0 $ and $ \Tau_1 $. Then we explain how to deduce the estimates of $ \Tau_2 $ following the ideas used for $ \Tau_1 $.

\item 
\paragraph*{Estimates for $  \Tau_0 $} We start by recalling some useful tools that has been already introduced in the poof of Lemma \ref{Lp:est}.
Let $ \left\{\psi_{\mu} \right\} $ a partition of unity defined as $ \psi_{\mu}(t) = \eta_{\mu}(t) - \eta_{\mu+1}(t) $. It holds  $ 0 \leq \psi_{\eta} \leq 1 $, $ \psi_{\mu}$ is smooth and supported in $ I_{\mu} = [\pi \mu, \pi(\mu +2)] $. Moreover take
$$ \psi^e_{\nu}  = \sum_{d = -1}^1 \psi_{\nu+d} = \eta_{\nu-1} - \eta_{\nu+1} = \eta_{\nu-1}(1-\eta_{\nu+1}) , $$ 
which is a set of smooth cut-off functions supported in the interior of $ I^e_{
	\nu} $, identically $ 1 $ in $ I_{\nu} $,  with uniformly bounded derivative and $ 0 \leq \psi_{\nu}^e \leq 1 $.  

For $ \nu > N+1 $, we estimate the $ l^p(\dot{\mathcal{H}}^{s}_{p}(I_{\nu})) $  seminorm of $ \Tau_0(R) $ by the seminorm of the extension $ \tilde{\Tau}_0(R) = \psi^{e}_{\nu}\Tau_0(R)$. We have
\begin{align*}
\|\Tau_0(R) \|_{l^p(\mathbb{Z}; \dot{\mathcal{H}}^{s}_{p}(I_{\nu}))}^p \leq  \sum_{k} \sum_{h} |h|^{sp } \left|[\psi^{e}_{\nu}(\Tau_0(R))_k]^{\widehat{ }}_h\right|^p = \sum_{k} \sum_{h \neq 0} \frac{1}{|h|^{(1-s) p}}\left|[\partial_{t} (\psi^{e}_{\nu}(\Tau_0(R))_k)]^{\widehat{ }}_h\right|^p.
\end{align*}
Let us compute the Fourier coefficients. First of all note that
\begin{equation*}
\partial_{t} (\psi^{e}_{\nu}(\Tau_0(R))_k) =  (\psi^{e}_{\nu})'(\Tau_0(R))_k + \psi^{e}_{\nu} (\Tau_0(R))'_k.
\end{equation*}
For the term $  (\psi^{e}_{\nu})'(\Tau_0(R))_k $, we use \eqref{dec:F0} and the proof of the estimate follows as for the analogous case in Lemma \ref{Lp:est}. Let now rewrite
\begin{align*}
\psi^{e}_{\nu} (\Tau_0(R))'_k = & \, - \frac{\psi^{e}_{\nu}(t)}{8 \pi t} \sum_{m \neq 0} \sum_{r(m)} e^{-itm} e^{i\frac{\Lambda_m}{8\pi}\log 4 t}
\alpha_{j_1}\bar{\alpha}_{j_2}\alpha_{j_3} \\
= & \, -\frac{1}{\nu+1}\sum_{m\neq 0}\sum_{r(m)} \sum_{l} e^{-i t (m-l/2)} [\tilde{\psi}^e_{\nu}]^{\widehat{ }}_{l} \alpha_{j_1}\bar{\alpha}_{j_2}\alpha_{j_3} \\
= & \, -\frac{1}{\nu+1} \sum_{h} e^{-ith/2}\left(\sum_{\substack{l \in \mathbb{Z} \\ l+h \text{ even} \\ l+h \neq 0 } }\sum_{r((h+l)/2)} [\tilde{\psi}^e_{\nu}]^{\widehat{ }}_{l} \alpha_{j_1}\bar{\alpha}_{j_2}\alpha_{j_3} \right). 
\end{align*}
Using the above equality and some H\"older inequalities, we deduce
\begin{align*}
\sum_{k} \sum_{h \neq 0} \frac{1}{|h|^{(1-s) p}}\left|[\psi^{e}_{\nu}(\Tau_0(R))_k']^{\widehat{ }}_h\right|^p =  & \, \frac{1}{(\nu+1)^p} \sum_{k} \sum_{h \neq 0} \frac{1}{|h|^{(1-s) p}} \left(\sum_{\substack{l \in \mathbb{Z} \\ l+h \text{ even} \\ l+h \neq 0 } }\sum_{r((h+l)/2)} \frac{ \langle l \rangle }{\langle l \rangle }[\tilde{\psi}^e_{\nu}]^{\widehat{ }}_{l} \alpha_{j_1}\bar{\alpha}_{j_2}\alpha_{j_3} \right)^p \\
\leq \, &  \frac{1}{(\nu+1)^p} \sum_{k} \sum_{h \neq 0} \frac{1}{|h|^{(1-s) p}} \left(\sum_{\substack{l \in \mathbb{Z} \\ l+h \text{ even} \\ l+h \neq 0 } }\sum_{r((h+l)/2)} \frac{ 1 }{\langle l \rangle^q }\right)^\frac{p}{q}    
\sum_{\substack{l \in \mathbb{Z} \\ l+h \text{ even} \\ l+h \neq 0 } }\sum_{r((h+l)/2)} \langle l \rangle^p \left|[\tilde{\psi}^e_{\nu}]^{\widehat{ }}_{l}\right|^p |\alpha_{j_1}\bar{\alpha}_{j_2}\alpha_{j_3}|^p \\
\lesssim &  \, \frac{1}{(\nu+1)^p} \sum_{k} \sum_{h \neq 0} \frac{h^{\eps}}{|h|^{(1-s) p}} 
\sum_{\substack{l \in \mathbb{Z} \\ l+h \text{ even} \\ l+h \neq 0 } }\sum_{r((h+l)/2)} \langle l \rangle^p \left|[\tilde{\psi}^e_{\nu}]^{\widehat{ }}_{l}\right|^p |\alpha_{j_1}\bar{\alpha}_{j_2}\alpha_{j_3}|^p \\
\lesssim & \,  \frac{1}{(\nu+1)^p} \|\alpha\|_{l^p}^3 \sup_{h\neq 0} \frac{1}{|h|^{(1-s)p - \eps}} \\
\lesssim & \,  \frac{1}{(\nu+1)^p} \|\alpha \|_{l^p}^3.
\end{align*} 
Note that the last inequality holds true only if we consider $ s < 1 $.

For $ \nu \in \{ N, N+1 \} $, as in Lemma \ref{Lp:est} we extend $ \Tau_0 $ by $ (1-\eta_{\nu+1})\Tau_0 $. For these terms the estimates follow easily as for $ \nu > N + 1 $.

\paragraph*{Estimates $ \Tau_1 $}

As before, we restrict only to the term $ F $ which is the first element of the right hand side of \eqref{MAMO}. For $ \nu > N+1 $, to estimate the $ l^p(\dot{\mathcal{H}}^{s}_{p}) $ seminorm we use the extension $ \tilde{F}(R) = \psi^{e}_{\nu}F(R)$. As before 
\begin{align*}
\|F(R) \|_{l^p(\mathbb{Z}; \dot{\mathcal{H}}^{s}_{p}(I_{\nu}))}^p \leq  \sum_{k} \sum_{h} |h|^{s p}\left|[\psi^{e}_{\nu}(F(R))_k]^{\widehat{ }}_h\right|^p = \sum_{k} \sum_{h \neq 0} \frac{1}{|h|^{(1-s) p}}\left|[\partial_{t} (\psi^{e}_{\nu}(F(R))_k)]^{\widehat{ }}_h\right|^p.
\end{align*}
Let us compute the Fourier coefficients. We use Leibnitz's rule as before and we use  \eqref{dec:F1} instead of \eqref{dec:F0}. Let now rewrite $ \psi^{e}_{\nu} (F(R))'_k $ as a Fourier series.

\begin{align*}
\psi^{e}_{\nu} (F(R))'_k = & \, - \sum_{d=-1}^1 \frac{\psi_{\nu+d}(t)}{8 \pi t} \sum_{m \neq 0} \sum_{r(m)} e^{-itm} e^{i\frac{\Lambda_m}{8\pi}\log 4 t} R_{j_1}\bar{\alpha}_{j_2}\alpha_{j_3} \\
= & \, -\sum_{d = -1}^1\frac{1}{(\nu+1+d)^{2}}\sum_{m\neq 0}\sum_{r(m)} \sum_{l} e^{-i t (m-l/2)} (\nu+1+d) [\tilde{\psi}_{\nu+d} R_{j_1}]^{\widehat{ }}_{l} \bar{\alpha}_{j_2}\alpha_{j_3}\\
= & \, -\sum_{d = -1}^1\frac{1}{(\nu+1+d)^{2}}\sum_{h} e^{-ith/2}\left(\sum_{\substack{l \in \mathbb{Z} \\ l+h \text{ even} \\ l+h \neq 0 } }\sum_{r((h+l)/2)} (\nu+1+d) [\tilde{\psi}_{\nu+d} R_{j_1}]^{\widehat{ }}_{l} \bar{\alpha}_{j_2}\alpha_{j_3} \right) \\
= & \, -\sum_{h} e^{-ith/2} \sum_{d = -1}^1\frac{1}{(\nu+1+d)^{2}} \left(\sum_{\substack{l \in \mathbb{Z} \\ l+h \text{ even} \\ l+h \neq 0 } }\sum_{r((h+l)/2)} (\nu+1+d) [\tilde{\psi}_{\nu+d} R_{j_1}]^{\widehat{ }}_{l} \bar{\alpha}_{j_2}\alpha_{j_3} \right). 
\end{align*}
For the $ l^p(\dot{\mathcal{H}}^{1-s}_{p}) $ seminorm of $ \psi^{e}_{\nu} (F(R))'_k $, we get 
\begin{align*}
\sum_{k} \sum_{h \neq 0} \frac{1}{|h|^{(1-s) p}}\left|[\psi^{e}_{\nu}(F(R))'_k]^{\widehat{ }}_h\right|^p  = & \, \sum_{k} \sum_{h \neq 0} \frac{1}{|h|^{(1-s) p}}\left|   \sum_{d = -1}^1\frac{1}{(\nu+1+d)^{2}} \left(\sum_{\substack{l \in \mathbb{Z} \\ l+h \text{ even} \\ l+h \neq 0 } }\sum_{r((h+l)/2)} (\nu+1+d) [\tilde{\psi}_{\nu+d} R_{j_1}]^{\widehat{ }}_{l} \bar{\alpha}_{j_2}\alpha_{j_3} \right)        \right|^p.
\end{align*}
Note that the sum in $ d $ is of three elements so it is enough to consider for example the case $  d = 0 $. After applying H\"older's inequality, we deduce
\begin{align*}
\frac{1}{(\nu+1)^{2p}} \sum_{k} \sum_{h \neq 0} &  \frac{1}{|h|^{(1-s) p}}\left| \sum_{\substack{l \in \mathbb{Z} \\ l+h \text{ even} \\ l+h \neq 0 } }\sum_{r((h+l)/2)} (\nu+1) [\tilde{\psi}_{\nu} R_{j_1}]^{\widehat{}}_{l} \bar{\alpha}_{j_2}\alpha_{j_3} \right|^p \\
\lesssim & \, \frac{1}{(\nu+1)^{2p}} \sum_{k} \sum_{h \neq 0}  \frac{1}{|h|^{(1-s) p-\eps}} \sum_{\substack{l \in \mathbb{Z} \\ l+h \text{ even} \\ l+h \neq 0 } }\sum_{r((h+l)/2)} \left| (\nu+1) \langle l \rangle^s [\tilde{\psi}_{2\nu} R_{j_1}]^{\widehat{ }}_{l} \bar{\alpha}_{j_2}\alpha_{j_3} \right|^p \\
\lesssim  & \, \frac{\|\alpha\|_{l^p}^{2p}}{(\nu+1)^{2p}}\|(\nu+1) R\|^p_{l^p(\mathbb{Z}; \mathcal{H}^s_{p}(I_{\nu}))} \sup_{h \neq 0 } \frac{1}{|h|^{(1-s)^p-\eps}},
\end{align*}
the desired estimates hold for $ s \in (0,1) $.

For $ \nu \in \{ N, N+1 \} $, as in Lemma \ref{Lp:est} we extend $ F $ by $ (1-\eta_{\nu+1})\Tau_0 $ and the estimates follow easily as for $ \nu > N + 1 $.

\paragraph*{Estimates of $ \Tau_2 $ }

As before let us restrict to the estimates of the cubic term $ G $.
For $ \nu > N+1 $, to estimate the $ l^p(\dot{\mathcal{H}}^{s}_{p}) $ seminorm we use the extension $ \tilde{G}(R) = \psi^{e}_{\nu}G(R)$. So
\begin{align*}
\|G(R) \|_{l^p(\mathbb{Z}; \dot{\mathcal{H}}^{s}_{p}(I_{\nu}))}^p \leq  \sum_{k} \sum_{h} |h|^{s p}\left|[\psi^{e}_{\nu}(G(R))_k]^{\widehat{ }}_h\right|^p = \sum_{k} \sum_{h \neq 0} \frac{1}{|h|^{(1-s) p}}\left|[\partial_{t} (\psi^{e}_{\nu}(G(R))_k)]^{\widehat{ }}_h\right|^p.
\end{align*}
Let us compute the Fourier coefficients of $\partial_{t} (\psi^{e}_{\nu}(G(R))_k)$. First of all note that
\begin{equation*}
\partial_{t} (\psi^{e}_{\nu}(G(R))_k) =  (\psi^{e}_{\nu})'(G(R))_k + \psi^{e}_{\nu} (G(R))'_k.
\end{equation*}
For the term $  (\psi^{e}_{\nu})'(G(R))_k $, we use \eqref{dec:F7} and the proof of the estimate follows as for the analogous case in Lemma \ref{Lp:est}. Similarly

\begin{align*}
\psi^{e}_{\nu} (G(R))'_k = & \, - \sum_{d=-1}^1 \frac{\psi_{\nu+d}(t)}{8 \pi t} \sum_{m \neq 0} \sum_{r(m)} e^{-itm} e^{i\frac{\Lambda_m}{8\pi}\log 4 t} R_{j_1}\bar{R}_{j_2}R_{j_3} \\
= & \, -\sum_{d = -1}^1\frac{1}{(\nu+1+d)^{4}}\sum_{m\neq 0}\sum_{r(m)} \sum_{l} e^{-i t (m-l/2)} (\nu+1+d)^{3} [\tilde{\psi}_{\nu+d} R_{j_1} \bar{R}_{j_2}R_{j_3}]^{\widehat{ }}_l \\
= & \, -\sum_{d = -1}^1\frac{1}{(\nu+1+d)^{4}}\sum_{h} e^{-ith/2}\left(\sum_{\substack{l \in \mathbb{Z} \\ l+h \text{ even} \\ l+h \neq 0 } }\sum_{r((h+l)/2)} (\nu+1+d) [\tilde{\psi}_{\nu+d} R_{j_1} \bar{R}_{j_2}R_{j_3}]^{\widehat{ }}_l\right) \\
= & \, -\sum_{h} e^{-ith/2} \sum_{d = -1}^1\frac{1}{(\nu+1+d)^{4}} \left(\sum_{\substack{l \in \mathbb{Z} \\ l+h \text{ even} \\ l+h \neq 0 } }\sum_{r((h+l)/2)} (\nu+1+d) [\tilde{\psi}_{\nu+d} R_{j_1} \bar{R}_{j_2}R_{j_3}]^{\widehat{ }}_l \right). 
\end{align*}
The estimates then follow straight-forward as in the case $ i = 1$.

Finally for the case $ \nu \in \{N, N+1 \}$. We extend $ G(R) $ by $ (1-\eta_{\nu+1}) G(R) $ and we argue as in the case $ \nu > N+1 $.

\end{proof}

\section{Proof of Proposition \ref{prop:exa}}
\label{sec:ex}

This section is dedicated to the proof of Proposition \eqref{prop:exa}. Let start by recalling equation \eqref{B:equ} which reads

\begin{equation*}
i\partial_{t} B_{k} = -\frac{1}{8\pi t} \sum_{m \neq 0}\sum_{r(m)} e^{-imt}e^{i\frac{\Lambda_m}{8\pi}\log 4t}B_{j_1}\bar{B}_{j_2}B_{j_3} + \frac{1}{8 \pi t}\left(|B_{k}|^2 - |\alpha_k|^2\right)B_{k} \quad  \text{ and } \quad \lim_{t \to +\infty} B_k = \alpha_k.
\end{equation*} 
We are looking for explicit solutions in the case the initial datum $ \alpha_k = \alpha $ for any $ k \in \mathbb{Z} $. First of all notice that $ \Lambda_m = 0 $ and let us make the ansatz that $ B_k(t) = B(t) $ for any $ k \in \mathbb{Z} $. The equations rewrite
\begin{equation*}
i\partial_{t} B = - |B|^2B \sum_{m \neq 0} \frac{e^{-imt}}{8\pi t}r_{m}  + \frac{1}{8 \pi t}\left(|B|^2 - |\alpha|^2\right)B \quad  \text{ and } \quad \lim_{t \to +\infty} B = \alpha.
\end{equation*} 
where we denote $ \sum_{r(m) } 1 = r_m $. Notice that by definition of $ r(m) $ we have $ r_m = r_{-m} $. We deduce that 
\begin{equation*}
\Im \left( \sum_{m\neq 0} e^{itm}r_m \right) =0,
\end{equation*} 
thus
\begin{equation*}
\partial_t|B(t)|^2=0.
\end{equation*}
The equation rewrites 
\begin{equation*}
i\partial_{t} B = - |\alpha|^2 B \sum_{m \neq 0} \frac{e^{-imt}}{8\pi t}r_{m}   \quad  \text{ and } \quad \lim_{t \to +\infty} B = \alpha.
\end{equation*} 
If we are able to show that $ \sum_{m\neq 0 } \frac{e^{-im \tau}}{ \tau}r_m $ is integrable in $ (T,+\infty ) $ for some $ T > 0 $, then the solution is 
\begin{equation*}
B(t) =  \alpha e^{-i|\alpha|^2\int_t^{+\infty} \sum_{m \neq 0}\sum_{r(m)} \frac{e^{-im \tau}}{8\pi \tau} \, d\tau} .
\end{equation*}
We have
\begin{equation*}
-\int_{t}^{+\infty} \sum_{m \neq 0} \frac{e^{-im \tau}}{ \tau} r_m  \, d\tau =   \sum_{m \neq 0} \frac{e^{-im t}}{ i m t} r_m  +  \int_{t}^{\infty} \sum_{m \neq 0} \frac{e^{-im \tau}}{ i m \tau^2} r_m  \, d\tau = \frac{1}{it} \sum_{m \neq 0} \frac{e^{-im \tau}}{  m} r_m + \sum_{m \neq 0 } r_m f(t,m),
\end{equation*}
with 
\begin{equation*}
|f(t,m)| \leq \frac{1}{m^2 t^2 }.
\end{equation*}
As 
$$r_m=2Card\, \{j\in\mathbb Z,\,j|m\},$$
we know that
$$r_m\lesssim \log m,$$
and the upper-estimate is optimal for $m=2^n$. Thus there is no issue for defining the part involving $f(x,m)$, and the remaining part
$$\tilde B(t):=e^{\frac{i}{t} \sum_{m\neq 0} e^{-im t} \frac{r_m}{im}},$$ 
satisfies
$$\tilde B\in l^\infty_\nu L^2(\nu,\nu+1)\cap l^p_\nu \dot H^{s}(\nu,\nu+1).$$
for all 
$$1\leq p\leq \infty,\quad 0\leq s<\frac 12,\quad (1+s)p>1.$$






\end{document}